\documentclass[11pt,reqno]{amsart}
\usepackage{amsmath}
\usepackage{amsfonts}
\usepackage{amssymb}
\usepackage{array}
\usepackage{caption}
\usepackage{latexsym}
\usepackage{color}
\usepackage[colorlinks=true, linkcolor=green, citecolor=magenta, urlcolor=blue]{hyperref}
\usepackage{url}
\usepackage{enumerate}
\usepackage{booktabs}
\usepackage{longtable}
\usepackage[figuresleft]{rotating}

\usepackage{multirow}
\usepackage{lscape}
\usepackage{tabularx}
\usepackage{adjustbox}
\usepackage{comment}

\renewcommand\a{\alpha}    

\newcommand\Ga{\mathrm{\Gamma}}

  \newcommand\att{\mathrm{Att}}       
       
\newcommand\K{\mathsf{K}}    \newcommand\N{\mathbf{N}}

\newcommand{\mz}{Z}

    \newcommand\Alt{\mathrm{Alt}}      \newcommand\Aut{\mathrm{Aut}}
   \newcommand\Cay{\mathrm{Cay}}      \newcommand\Cos{\mathrm{Cos}}

\newcommand\AGL{\mathrm{AGL}}

\newcommand\PSL{\mathrm{PSL}}

\newtheorem{theorem}{Theorem}[section]%
\newtheorem{lemma}[theorem]{Lemma}%
\newtheorem{corollary}[theorem]{Corollary}%
\newtheorem{problem}[theorem]{Problem}%
\newtheorem{question}[theorem]{Question}%
\theoremstyle{definition}
\newtheorem{example}[theorem]{Example}%

\begin{document}

\title[On automorphism groups of half-arc-transitive tetravalent graphs]
{On automorphism groups of half-arc-transitive tetravalent graphs}

\thanks{$^*$Corresponding author. }
\thanks{2020 MR Subject Classification 05C25, 05E18, 20B25.}
%\thanks{This work was partially supported by the National Natural Science Foundation of China (11731002, 12071023) and the 111 Project of China (B16002).}

\author{Yuan-Dong Li} \author{Binzhou Xia} \author{Jin-Xin Zhou$^*$}
\address{School of mathematics and statistics\\
Beijing Jiaotong University\\
Beijing \\
100044, P. R. China}
\email{23111509@bjtu.edu.cn (Yuan-Dong Li)}

\address{School of Mathematics and Statistics\\The University of Melbourne\\Parkville, VIC 3010\\Australia}
\email{binzhoux@unimelb.edu.au (Binzhou Xia)}

\address{School of mathematics and statistics\\
Beijing Jiaotong University\\
Beijing \\
100044, P. R. China}
\address{Beijing Key Laboratory of Biological Big Data and Topological Statistics\\
Beijing Jiaotong University\\
Beijing\\
100044, P.R. China}
\email{jxzhou@bjtu.edu.cn (Jin-Xin Zhou)}

\begin{abstract}
We characterize connected tetravalent graphs $\Gamma$ which admit groups $M<H$ of automorphisms such that $\Gamma$ is $M$-half-arc-transitive and $H$-arc-transitive. Examples for each case in our characterization are constructed, including a counter-example to a question asked by A. Ramos Rivera and P.~\v{S}parl in 2019 as well as the first example of a tetravalent normal-edge-transitive non-normal Cayley graph on a nonabelian simple group.

\end{abstract}

\maketitle

\section{Introduction}

All graphs considered in this paper are finite, undirected, regular, connected and simple. Let $\Ga$ be a graph. For a positive integer $s$, a sequence $(u_0, u_1, \ldots, u_s)$ of vertices of $\Ga$ is said to be an \emph{$s$-arc} if $u_i$ is adjacent to $u_{i+1}$ for $0\leq i\leq s-1$ and $u_{j-1}\neq u_{j+1}$ for $1\leq j\leq s-1$. A $1$-arc is simply an \emph{arc}. We use $V(\Gamma)$, $E(\Gamma)$, $A(\Ga)$ and $\Aut(\Gamma)$ to denote the vertex set, edge set, arc set and (full) automorphism group of $\Gamma$, respectively.

Let $G$ be a subgroup of $\Aut(\Ga)$. If $G$ is transitive on $V(\Ga)$ and on $E(\Ga)$ but not on $A(\Ga)$, then we say that $\Ga$ is \emph{ $G$-half-arc-transitive}. In the literature, ``half-arc-transitive'' is also abbreviated as ``$\frac{1}{2}$-arc-transitive'' or ``HAT''. When $\Ga$ is $\Aut(\Ga)$-half-arc-transitive, we omit the prefix and simply say that $\Ga$ is half-arc-transitive (or HAT).
The study of HAT graphs was initiated by Tutte in \cite{Tutte1966}. Over the past half-century, substantial progress has been made, and we refer the reader to the survey articles \cite{CPS2015,Marusic1998} for an overview. Much of the work on HAT graphs has focused on the tetravalent case; see, for example, \cite{M-DM,MM1999,MN-JGT,MP1999,MX1997,PW-JCTB2007,RSparl2019,Spiga-Xia2021,Xia-2021,Zhou-JACO,Zhou-PAMS}.

For a positive integer $t$, we say that $\Ga$ is \emph{$(G,t)$-arc-transitive} if $G$ is transitive on the set of $t$-arcs of $\Ga$. If such a subgroup $G$ exists, then a $(G,t)$-arc-transitive graph will be simply called a \emph{$t$-arc-transitive graph}.
For $s\in\mathbb{N}\cup\{\frac{1}{2}\}$, if $\Ga$ is $(G,s)$-arc-transitive but not $(G,\lfloor s+1\rfloor)$-arc-transitive, then we say that $\Ga$ is \emph{$(G,s)$-transitive}.
A pair $(M, H)$ of subgroups of $\Aut(\Ga)$ is called an \emph{$(s,t)$-pair} of $\Ga$ if $M\leq H\leq\Aut(\Ga)$ and $\Ga$ is $(M,s)$-transitive and $(H,t)$-transitive. Such a pair $(M, H)$ is called \emph{maximal} if $M$ is a maximal subgroup of $H$.
Note that, $\Gamma$ admits a maximal $\bigl(\frac{1}{2},t\bigr)$-pair for some positive integer $t$ if and only if $\Gamma$ is $M$-half-arc-transitive and $H$-arc-transitive for some groups $M<H$ of automorphisms.
The following problem is both natural and of considerable interest.

\begin{problem}\label{prob}
Given a graph $\Ga$, classify maximal $(s,t)$-pairs of $\Ga$ for all $s\in\mathbb{N}\cup\{\tfrac{1}{2}\}$ and $t\in\mathbb{N}$ with $s<t$.
\end{problem}

In this paper we address Problem~\ref{prob} for tetravalent graphs in the case $s=\tfrac{1}{2}$.
Our motivation is to better understand the automorphism groups of tetravalent $G$-HAT graphs $\Gamma$.
Deciding whether a given $G$-HAT graph $\Gamma$ is HAT is often difficult, largely because determining $\Aut(\Gamma)$ is hard in general.
If $\Gamma$ is not half-arc-transitive, then $\Aut(\Gamma)$ is arc-transitive; consequently, $\Gamma$ admits a maximal $\bigl(\frac{1}{2},t\bigr)$-pair $(M,H)$ for some positive integer $t$ with $G\le M$.

Another motivation comes from studying the symmetry of the graph of alternating cycles associated with a tetravalent $G$-HAT graph.
Let $\Ga$ be a tetravalent $M$-HAT graph for some $M\leq\Aut(\Ga)$. It is easy to see that the action of
$M$ on the arc set of $\Ga$ has two orbits, denoted by $O_M^+(\Ga)$ and $O_M^-(\Ga)$. For each $\epsilon\in \{+,-\}$,
and each edge of $\Ga$, the orbit $O_M^\epsilon(\Ga)$ contains exactly one of the two arcs corresponding to this edge.
Hence $O_M^+(\Ga)$ and $O_M^-(\Ga)$ give two opposite orientations of the edges of $\Ga$ and hence give two digraphs with the same vertex set as $\Ga$, denoted by $D_M^+(\Ga)$ and $D_M^-(\Ga)$.
A cycle of $\Ga$ is called an \emph{$M$-alternating cycle} if consecutive edges along the cycle have opposite orientations in $D_M^+(\Ga)$ (and hence also opposite orientations in $D_M^-(\Ga)$). In a seminal paper~\cite{Marusic1998} in 1998, Maru\v si\v c proved that all alternating cycles of the tetravalent $M$-HAT graph $\Ga$ have the same length, half of which is called the \emph{$M$-radius} of $\Ga$, and any two alternating cycles of $\Ga$ with nonempty intersection share the same number of vertices. This number is called the \emph{$M$-attachment number} of $\Ga$, denoted by $\att_M(\Ga)$.

The \emph{graph of $M$-alternating cycles} of $\Ga$, denoted $\Alt_M(\Ga)$, is the graph whose vertices are the $M$-alternating cycles of $\Ga$, with two such cycles adjacent whenever they have a non-empty intersection (see~\cite{MP1999,PSparl2017}). In~\cite{RSparl2019}, {Ramos} Rivera and \v{S}parl initiated the study of the automorphism group of $\Alt_M(\Ga)$ for tetravalent $M$-HAT graphs $\Ga$.
They proved that, if either $\att_M(\Ga)$ equals the $M$-radius or $\Aut(\Ga)$ contains an element interchanging $O_M^+(\Ga)$ and $O_M^-(\Ga)$, then $\Alt_M(\Ga)$ is arc-transitive (see~\cite[Proposition~5.7]{RSparl2019}). However, a tetravalent arc-transitive $M$-HAT graph need not admit an automorphism interchanging $O_M^+(\Ga)$ and $O_M^-(\Ga)$, and so this criterion does not always apply. Despite this, all computed examples in~\cite{RSparl2019} still have $\Alt_M(\Ga)$ arc-transitive. This led the authors to ask the following question concerning potential weaker conditions ensuring the arc-transitivity of $\Alt_M(\Ga)$, and similarly for the related quotient graph $\Ga_{\mathcal{B}(M)}$ (with respect to a certain $M$-invariant partition $\mathcal{B}(M)$ of $V(\Ga)$; see~\cite{MP1999} or~\cite{RSparl2019} for definition). In particular, $\Ga_{\mathcal{B}(M)}=\Ga$ whenever $\att_M(\Ga)=1$.

\begin{question}[{\cite[Question~5.8]{RSparl2019}}]\label{que:Sparl}
Is it true that, if $\Ga$ is a tetravalent arc-transitive graph which is $M$-HAT for some $M\leq\Aut(\Ga)$, the graphs $\Alt_M(\Ga)$ and $\Ga_{\mathcal{B}(M)}$ are both arc-transitive?
\end{question}

Note that any larger HAT group $N>M$ has the same orbits on $A(\Gamma)$ with $M$, hence induces the same graph of alternating cycles. We may let $M$ be maximal subject to the condition that $\Ga$ is $M$-HAT. Take $H\leq\Aut(\Ga)$ such that $M$ is maximal in $H$. Then $H$ is arc-transitive on $\Ga$, and so $(M, H)$ is a maximal $(\frac{1}{2}, t)$-pair. If $M$ is normal in $H$, then $H$ preserves $\{O_M^+(\Ga),O_M^-(\Ga)\}$,  and so we deduce from~\cite[Proposition~5.7]{RSparl2019} that $\Alt_M(\Ga)$ is arc-transitive. Thus, the question is reduced to the case that $M$ is non-normal in $H$.

Our main theorem characterizes tetravalent graphs admitting a maximal $(\frac{1}{2}, t)$-pair with $t\geq1$. Recall that, for a subgroup $X$ of a group $Y$, the \emph{core} of $X$ in $Y$, denoted by $\mathrm{Core}_Y(X)$, is the largest normal subgroup of $Y$ contained in $X$.
Let $\Ga$ be a graph. Assume that $G\leq\Aut(\Ga)$ is such that $G$ is transitive on the vertex set of $\Ga$. Let $N$ be a normal subgroup of $G$ such that $N$ is intransitive on $V(\Ga)$. The \emph{normal quotient graph\/} $\Ga_N$ of $\Ga$ relative to $N$ is defined as the graph whose vertices are the orbits of $N$ on $V(\Ga)$ and where two different orbits are adjacent if there exists an edge in $\Ga$ between the vertices lying in those two orbits. If $\Ga_N$ and $\Ga$ have the same valency, then $\Ga$ is a \emph{cover} of $\Ga_N$.

\begin{theorem}\label{Thm}
Let $\Ga$ be a tetravalent graph, let $(M, H)$ be a maximal $(\frac{1}{2}, t)$-pair of $\Ga$ with $t\geq1$, and let $K=\mathrm{Core}_H(M)$. Take $u\in V(\Ga)$. Then one of the following holds.
\begin{enumerate}[\rm(a)]
\item\label{Thmcase:t=3} $t=3$, $(H/K, M/K, H_u, M_u)=(A_{72}, A_{71}, S_4\times S_3, \mz_2)$, and $\Ga$ is a cover of $\Ga_K$.
\item\label{Thmcase:t=2} $t=2$, $(H/K, M/K, H_u, M_u)= (S_5, \AGL_1(5), A_4, \mz_2)$, and $\Ga$ is a cover of $\Ga_K\cong \K_{5, 5}-5\K_2$.
\item\label{Thmcase:t=1} $t=1$, and either $M$ is normal in $H$, or $M_u^{\Ga(u)}\cong\mz_2$ is non-normal in $H_u^{\Ga(u)}\cong D_8$. {Moreover, in the latter case, $K$ is semiregular on $V(\Ga)$, $|V(\Ga_K)|$ is not a power of $2$, and one of the following holds:}
\begin{enumerate}[\rm (c1)]
\item\label{Thmcase:t=1_cover} $K$ is the kernel of $H$ acting on $V(\Ga_K)$, $\Ga$ is a cover of $\Ga_K$, and the pair $(M/K, H/K)$ is a maximal $(\frac{1}{2}, 1)$-pair of $\Ga_K$;
\item\label{Thmcase:t=1_cycle} $K$ is the kernel of $M$ acting on $V(\Ga_K)$, $\Ga_K\cong C_r$, and $M/K\cong D_{2r}$.
\end{enumerate}
\end{enumerate}
Moreover, examples for each of the above cases exist.
\end{theorem}

{We summarize the proof of Theorem~\ref{Thm} in Section~\ref{sec:sum}, and present the full details in Sections~\ref{sec:proofofmain} and~\ref{sec:examples}.} More precisely, we first prove in Section~\ref{sec:proofofmain} that {under the assumptions of the theorem,} one of~\eqref{Thmcase:t=3}--\eqref{Thmcase:t=1} must occur, and then in Section~\ref{sec:examples} we construct examples for each of the cases~\eqref{Thmcase:t=3},~\eqref{Thmcase:t=2},~\eqref{Thmcase:t=1_cover} and~\eqref{Thmcase:t=1_cycle}. The example for~\eqref{Thmcase:t=1_cover} yields the following corollary, which provides a negative answer to Question~\ref{que:Sparl} of {Ramos} Rivera and \v{S}parl~\cite{RSparl2019}.

\begin{corollary}\label{cor-1}
There exists a tetravalent graph $\Ga$ admiting a maximal $(\frac{1}{2}, 1)$-pair $(M,\Aut(\Gamma))$ such that $\Alt_M(\Ga)$ is half-arc-transitive.
\end{corollary}

Another byproduct of proving Theorem~\ref{Thm} is the discovery of the first example of a tetravalent normal-edge-transitive but non-normal Cayley graph on a nonabelian simple group.
Given a group $G$ and an inverse-closed generating subset $S\subseteq G\setminus\{1\}$, the \emph{Cayley graph} $\Cay(G,S)$ on $G$ is the graph with vertex set $G$ and edge set $\{\{g, sg\} \mid g\in G, s\in S\}$. Denote by $R_G$ the right multiplication action of $G$ on itself, and denote
\[
\Aut(G, S)=\{\a\in\Aut(G)\mid S^\a=S\}.
\]
It is well known~\cite{Godsil1981} that the normalizer of $R_G(G)$ in $\Aut(\Cay(G,S))$ is $R_G(G)\rtimes\Aut(G,S)$. We call $\Cay(G,S)$ \emph{normal-edge-transitive} if this normalizer is transitive on $E(\Cay(G,S))$, and \emph{normal} if this normalizer equals $\Aut(\Cay(G,S))$.
In 1999, Praeger~\cite{Praeger1999} proved that every normal-edge-transitive Cayley graph of valency $3$ on a nonabelian simple group is normal. This result was recently extended to all prime valencies in~\cite{Zhang-F-Z-Y}, where it was also shown that there exist normal-edge-transitive non-normal Cayley graphs of valency $8$ on nonabelian simple groups. However, it has remained open whether such graphs exist at smaller valencies, particularly valency $4$, the smallest possible valency, despite substantial work on tetravalent edge-transitive Cayley graphs on nonabelian simple groups~\cite{FangLiXu,FangWZ}. Our example for Theorem~\ref{Thm}\,\eqref{Thmcase:t=3} resolves this question by establishing the following:

\begin{corollary}\label{cor-2}
There exists a tetravalent normal-edge-transitive non-normal Cayley graph on $A_{71}$.
\end{corollary}

\section{Preliminaries}

For a positive integer $n$ and prime $p$, let $n_p$ denote the largest $p$-power dividing $n$. The $n$-cycle graph is denoted by $C_n$, the cyclic group of order $n$ is denoted by $\mz_n$, and the dihedral group of order $2n$ is denoted by $D_{2n}$.

Let $G$ be a permutation group on a set $\Omega$. If $G$ stabilizes a subset $\Delta\subseteq \Omega$, then denote by $G^\Delta$ the permutation group on $\Delta$ induced by $G$.  For $v\in \Omega$, the stabilizer in $G$ of $v$ is
the subgroup $G_v:=\{g\in G \mid v^g=v\}$ of $G$. If $G_v$ is trivial for every $v\in\Omega$, then we say that $G$ is \emph{semiregular} on $\Omega$, and that it is \emph{regular} if in addition $G$ is transitive on $\Omega$.

Let $\Ga$ be a graph and let $G\leq\Aut(\Ga)$. The neighborhood of a vertex $v\in V(\Ga)$ is denoted by $\Ga(v)$, and the kernel of the natural homomorphism $G_v\to G_v^{\Ga(v)}$ is denoted by $G_v^{[1]}$. If $G$ is transitive on $V(\Ga)$, $E(\Ga)$ or $A(\Ga)$, then we say that $\Ga$ is \emph{$G$-vertex-transitive}, \emph{$G$-edge-transitive} or \emph{$G$-arc-transitive}, respectively.

\begin{lemma}[{\cite[Theorem~1.1]{AAMPS2016}}]\label{lem:4HAT}
Let $\Ga$ be a tetravalent $G$-HAT graph and let ${K}\trianglelefteq G$ such that $K$ has at least three orbits on $V(\Ga)$. Then one of the following holds:
\begin{enumerate}[{\rm (a)}]
  \item $\Ga$ is a cover of $\Ga_K$, $K$ is semiregular on $V(\Ga)$ and is the kernel of $G$ acting on $V(\Ga_K)$, and {$\Gamma_K$} is $G/K$-HAT.
  \item $\Ga_K$ is an $r$-cycle, and {the permutation group induced by $G$} on $V(\Ga_K)$ is $\mz_r$ or $D_{2r}$.
\end{enumerate}
\end{lemma}

For a graph $\Gamma$ and a subgroup $G$ of $\Aut(\Ga)$, we say that $\Ga$ is \emph{$G$-locally-primitive} if $G_v^{\Ga(v)}$ is primitive for each $v\in V(\Gamma)$.

\begin{lemma}[{\cite[Theorem~4.1]{Praeger1993}}]\label{lem:quot}
Let $\Ga$ be a $G$-vertex-transitive and $G$-locally-primitive graph, and let $N\trianglelefteq G$ such that $N$ has at least three orbits on $V(\Ga)$. Then the following statements hold:
\begin{enumerate}[{\rm (a)}]
  \item $\Ga$ is a cover of $\Ga_N$.
  \item $N$ is semiregular on $V(\Ga)$ and is the kernel of $G$ acting on $V(\Ga_N)$.
  \item $\Ga_N$ is $(G/N)$-locally-primitive.
  \item If $\Ga$ is $(G,2)$-arc-transitive, then $\Ga_N$ is $(G/N,2)$-arc-transitive.
\end{enumerate}
\end{lemma}

Let $G$ be a group, let $H$ be a subgroup of $G$, and let $D$ {be an inverse-closed} union of double cosets of $H$ in $G$ such that $\langle H,D\rangle=G$. The \emph{coset graph} $\Cos(G, H, D)$ is the graph with vertex set $[G : H] = \{Hx \mid x \in G\}$ such that $Hx$ and $Hy$ are adjacent if and only if $yx^{-1}\in D$.
The following is a well-known and easy-to-prove result linking vertex- and edge-transitive graphs to coset graphs, for which we refer the reader to~\cite[Lemma~2.1]{Li-Lu-Zhang-JCTB}. Note that, for each $G$-vertex-transitive graph $\Gamma$ and $v\in V(\Gamma)$, there exists $g\in G$ such that $\{v, v^g\}\in E(\Ga)$.

\begin{lemma}\label{lem:cosetgraph1}
Let $\Ga$ be a $G$-vertex-transitive and $G$-edge-transitive graph, let $v \in V(\Ga)$, and {let $g\in G$ be} such that $\{v, v^g\}\in E(\Ga)$. Then the following statements hold:
\begin{enumerate}[{\rm (a)}]
  \item $\Ga \cong \Cos(G, G_v, G_v\{g, g^{-1}\}G_v)$.
  \item $\Ga$ is $G$-arc-transitive if and only if $G_v\{g, g^{-1}\}G_v=G_vgG_v$.
  \item If $\Ga$ is $G$-arc-transitive, then $\Ga \cong \Cos(G, G_v, G_vhG_v)$ for some element $h$ of {order a power of $2$} in $G$ reversing an arc $(u,v)$.
\end{enumerate}
\end{lemma}

A \emph{Haar graph} of a group $G$ is a bipartite graph whose automorphism group has a subgroup isomorphic to $G$ that is semiregular on the vertex set with orbits giving {the} bipartition.

\begin{lemma}[{\cite[Lemma~2.3]{DPZ}}]\label{lem:faithful}
Let $\Ga$ be a Haar graph of a group $G$, and let $N=\N_{\Aut(\Ga)}(G)$. Then for each $v\in V(\Ga)$, the stabilizer $N_v$ acts faithfully on $\Ga(v)$.
\end{lemma}

\section{Characterization of maximal \texorpdfstring{$(\frac{1}{2}, t)$}{(1/2, t)}-pairs}\label{sec:proofofmain}

Recall that the purpose of Theorem~\ref{Thm} is to characterize maximal $(\frac{1}{2}, t)$-pairs of tetravalent graphs for $t\geq1$. In the following lemma, we first deal with the case $t\geq2$.

\begin{lemma}\label{lem:hat-2at}
Let $\Ga$ be a tetravalent graph, let $u\in V(\Ga)$, let $(M, H)$ be a maximal $(\frac{1}{2}, t)$-pair of $\Ga$ with $t\geq2$, and let $K=\mathrm{Core}_H(M)$. Then one of the following holds:
{
\begin{enumerate}[\rm (a)]
\item $t=3$, $(H/K, M/K, H_u, M_u)=(A_{72}, A_{71}, S_4\times S_3, \mz_2)$, and $\Ga$ is a cover of $\Ga_K$;
\item $t=2$, $(H/K, M/K, H_{u}, M_{u})=(S_5, \AGL_1(5), A_4, \mz_2)$, and $\Ga$ is a cover of $\Gamma_K\cong \K_{5, 5}-5\K_2$.
\end{enumerate}}
\end{lemma}

\begin{proof}
By our assumption, $\Ga$ is $M$-HAT and $(H, 2)$-arc-transitive, and $M$ is a maximal subgroup of $H$. Take an arbitrary $v\in V(\Ga)$. Then $H_v^{\Ga(v)}$ is $2$-transitive as $H$ is $2$-arc-transitive. Since $K\trianglelefteq H$, it follows that $K_v\trianglelefteq H_v$ and hence $K_v^{\Ga(v)}\trianglelefteq H_v^{\Ga(v)}$. Consequently, the normal subgroup $K_v^{\Ga(v)}$ of $H_v^{\Ga(v)}$ is either trivial or transitive. Since $K\leq M$ and $M$ is not arc-transitive, $K_v^{\Ga(v)}$ is not transitive. Thus, $K_v^{\Ga(v)}$ is trivial, which implies that $K_v=1$, and so $K$ is semiregular on $V(\Ga)$.

Suppose {first that} $K$ is transitive on $V(\Ga)$. Then $H=K\rtimes H_u$. Since $M$ is HAT, $M_u$ is a $2$-group. Let $P$ be a Sylow $2$-subgroup of $H_u$ containing $M_u$. Then $M=K\rtimes M_u\leq K\rtimes P\leq H$. {Note that each Sylow $2$-subgroup of a transitive permutation group of degree $4$ is again transitive. Since $H_u$ is transitive on $\Ga(u)$, we deduce that $P$ is transitive on $\Ga(u)$.} Since $M_u$ is not transitive on $\Ga(u)$, we deduce that $M_u<P$ and so $M=K\rtimes M_u<K\rtimes P$. Since $H_u^{\Ga(u)}$ is $2$-transitive, its order is divisible by $3$. Thus, $|H_u|$ is divisible by $3$, and so $P<H_u$, which implies that $K\rtimes P<K\rtimes H_u=H$. Now $M<K\rtimes P<H$, {contradicting the maximality} of $M$ in $H$.

Suppose {next that} $K$ has precisely two orbits on $V(\Ga)$. Then since $H$ is arc-transitive on $\Ga$, the two orbits of $K$ are both independent sets of $\Ga$. By Lemma~\ref{lem:faithful}, $H_u$ is faithful on $\Ga(u)$ and so is $M_u$. {Since $H_u$ is 2-transitive on $\Gamma(u)$ and $M_u$ has two orbits of length $2$ on $\Ga(u)$, we deduce that either $H_u=A_4$ and $M_u=\mz_2$, or $H_u=S_4$ and $M_u\leq\mz_2^2$.}
In particular, $|H_u|_2>|M_u|$.
{Since $M_u$ is a $2$-group, there exists a Sylow $2$-subgroup $P$ of $M$ containing $M_u$. Note that $|M:K|=2|M_u|$ is a power of $2$, so $M/K$ is a $2$-group and hence $PK=M$. In particular, {$PM/K=M/K$.} Since $M$ is vertex-transitive, $M/K$ acts nontrivially on the set of the two $K$-orbits, and so $P$ contains an element interchanging them. Let $Q$ be a Sylow $2$-subgroup of $H$ containing $P$. Then $M=KP\leq KQ$.
Since $|H|=2|K||H_u|$ and $|H_u|$ is divisible by $3$, we obtain $KQ<H$.
This forces $M=KQ$ by the maximality of $M$ in $H$. Hence $|M|_2=|KQ|_2=|H|_2$, contradicting $|H|_2/|M|_2\geq|H_u|_2/|M_u|_2>1$.}

Thus, we conclude that $K$ has at least three orbits on $V(\Ga)$. Since $H$ is $2$-arc-transitive on $\Ga$, it follows from Lemmas~\ref{lem:4HAT} and~\ref{lem:quot} that $\Ga_K$ is $(M/K)$-HAT and $(H/K,2)$-arc-transitive, and $\Ga$ is a cover of $\Ga_K$. In particular, $K$ is the kernel of $H$ acting on $V(\Ga_K)$.

Let $\overline{\phantom{w}}$ be the natural homomorphism from $H$ to $H/K$, let $\Omega$ be the set of right cosets of $\overline{M}$ in $\overline{H}$, and let $n=|\Omega|$. {Then, since $\overline{M}$ is maximal and core-free in $\overline{H}$, the group} $\overline{H}$ acts faithfully and primitively on $\Omega$, and we may view $\overline{H}$ as a subgroup of $S_n$.

Since $\overline{M}$ is transitive on $V(\Ga_K)$, it follows that $\overline{H}=\overline{M}\,\overline{H}_{\overline{u}}$, where $\overline{u}:=u^K$. Then the subgroup $\overline{H}_{\overline{u}}$ of $\overline{H}$ acts faithfully and transitively on $\Omega$ with stabilizer $\overline{M}\cap \overline{H}_{\overline{u}}=\overline{M}_{\overline{u}}$. This implies that $\overline{M}_{\overline{u}}$ is core-free in $\overline{H}_{\overline{u}}$. Since $\Ga_K$ is $\overline{M}$-HAT, there are two orbits of $\overline{M}_{\overline{u}}$ on $\Ga_K(\overline{u})$, each of size $2$. As a consequence, $2|\overline{M}_{\overline{u}}|\leq|\overline{H}_{\overline{u}}|_2$.
Take $v\in\Ga(u)$ and let $\overline{v}=v^K$. Then $\overline{v}\in\Ga_K(\overline{u})$, and the action of $\overline{H}_{\overline{u}}$ on $\Ga_K(\overline{u})$ is equivalent to the action of $\overline{H}_{\overline{u}}$ on the set $\Delta:=\{\overline{H}_{\overline{u}\,\overline{v}}\overline{h}\mid \overline{h}\in\overline{H}_{\overline{u}}\}$.
Hence $\overline{M}_{\overline{u}}$ lies in
\[
\mathcal{X}:=\{X\leq\overline{H}_{\overline{u}}\mid X\text{ is core-free in }\overline{H}_{\overline{u}},\ 2|X|\leq|\overline{H}_{\overline{u}}|_2,\ X\text{ has two orbits of size $2$ on }\Delta\}.
\]
Since $\overline{H}$ is $2$-arc-transitive on $\Ga_K$, by \cite[Theorem~4]{Primoz}, $(\overline{H}_{\overline{u}}, \overline{H}_{\overline{u}\,\overline{v}})$ is one of {the} pairs $(L, B)$ listed in Table~\ref{tab:amalgams}. 

\begin{table}[htbp]
\centering
\caption{Finite faithful $2$-transitive amalgams of index $(4,2)$}
\label{tab:amalgams}
\tiny
\renewcommand{\arraystretch}{1.2}
\begin{adjustbox}{width=\textwidth,center}
\begin{tabular}{>{\raggedright\arraybackslash}p{1.5cm}>{\raggedright\arraybackslash}p{12cm}}
\hline
\textbf{Name} & $G = L *_B R$ \\
\hline
$A_4 s$ & $G = \langle x, y, s, a \mid x^2, y^2, s^3, a^2, [x, y], x^s = y, y^s = xy, s^a = s^{-1} \rangle$ \\
& $L = \langle x, y, s \rangle \cong A_4$, $B = \langle s \rangle \cong \mz_3$, $R = \langle s, a \rangle \cong S_3$ \\[0.3em]
$A_4 x$ & $G = \langle x, y, s, a \mid x^2, y^2, s^3, a^2, [x, y], x^s = y, y^s = xy, [s, a] \rangle$ \\
& $L = \langle x, y, s \rangle \cong A_4$, $B = \langle s \rangle \cong \mz_3$, $R = \langle s, a \rangle \cong \mz_3 \times \mz_2$ \\[0.3em]
$S_4$ & $G = \langle x, y, s, t, a \mid x^2, y^2, s^3, t^2, a^2, [x, y], s^t = s^{-1}, x^s = y, y^s = xy, x^t = y, [s, a], [t, a] \rangle$ \\
& $L = \langle x, y, s, t \rangle \cong S_4$, $B = \langle s, t \rangle \cong S_3$, $R = \langle s, t, a \rangle \cong S_3 \times \mz_2$ \\[0.3em]
$\mz_3 \times A_4$ & $G = \langle x, y, c, d, a \mid x^2, y^2, c^3, d^3, a^2, [x, y], [c, d], [c, x], [c, y], x^d = y, y^d = xy, c^a = d \rangle$ \\
& $L = \langle x, y, c, d \rangle \cong \mz_3 \times A_4$, $B = \langle c, d \rangle \cong \mz_3^2$, $R = \langle c, d, a \rangle \cong \mz_3^2 \rtimes \mz_2$ \\[0.3em]
$\mz_3 \rtimes S_4$ & $G = \langle x, y, c, d, t, a \mid x^2, y^2, c^3, d^3, t^2, a^2, [x, y], [c, d], [c, x], [c, y], c^t = c^{-1}, d^t = d^{-1}, x^d = y, y^d = xy, x^t = y,c^a = d, [a, t] \rangle$ \\
& $L = \langle x, y, c, d, t \rangle \cong \mz_3 \rtimes S_4$, $B = \langle c, d, t \rangle \cong \mz_3^2 \rtimes \mz_2$, $R = \langle c, d, t, a \rangle \cong \mz_3^2 \rtimes \mz_2^2$ \\[0.3em]
$\mz_3 \rtimes S_4^*$ & $G = \langle x, y, c, d, t, a \mid x^2, y^2, c^3, d^3, t^2, a^2 = t, [x, y], [c, d], [c, x], [c, y], c^t = c^{-1}, d^t = d^{-1}, x^d = y, y^d = xy,x^t = y, c^a = d, d^a = c^{-1} \rangle$ \\
& $L = \langle x, y, c, d, t \rangle \cong \mz_3 \rtimes S_4$, $B = \langle c, d, t \rangle \cong \mz_3^2 \rtimes \mz_2$, $R = \langle c, d, t, a \rangle \cong \mz_3^2 \rtimes \mz_4$ \\[0.3em]
$S_3 \times S_4$ & $G = \langle x, y, c, d, r, s, a \mid x^2, y^2, c^3, d^3, r^2, s^2, a^2, [x, y], [c, d], [r, s], [c, x], [c, y], c^r = c^{-1}, [d, r], [c, s], d^s = d^{-1}, x^d = y, y^d = xy, x^s = y, [r, x], [r, y], c^a = d, s^a = r \rangle$ \\
& $L = \langle x, y, c, d, r, s \rangle \cong S_3 \times S_4$, $B = \langle c, d, r, s \rangle \cong S_3^2$, $R = \langle c, d, r, s, a \rangle \cong S_3^2 \rtimes \mz_2$ \\[0.3em]
4-AT & $G = \langle t, c, d, e, x, y, a \mid t^2, c^3, d^3, e^3, x^2, y^2, a^2, [c, d], [c, e], [d, e] = c,[x, y], (cx)^2, (dx)^2,$ $[e, x], (cy)^2, [d, y], (ey)^2, c^t = d^{-1}, y(et)^2 e^{-1} te^{-1}, (et)^4 x, (ca)^2, d^a = e, x^a = y \rangle$ \\ &$L = \langle t, x, y, c, d, e \rangle \cong \AGL_2(3)$, $B = \langle x, y, c, d, e \rangle \cong \mz_3^2 \rtimes D_6$, $R = \langle x, y, c, d, e, a \rangle \cong (\mz_3^2 \rtimes D_6) \rtimes \mz_2$ \\[0.3em]
7-AT & $G = \langle h, p, q, r, s, t, u, v, k, a \mid h^4, p^3, q^3, r^3, s^3, t^3, u^3, v^2, k^2, a^2, kh^2, [p, q], [p, r], [p, s], [p, t],$ $[p, u], [q, r], [q, s], [q, t], [q, u], [r, s], [r, t], [u, s], [s, t]p^{-1}, [u, r]q^{-1}, [t, u](qrsp^{-1})^{-1}, [k, v], (tk)^2,$ $(rk)^2, [p, k], (qk)^2, (sk)^2, [u, k], (tv)^2, [r, v], (pv)^2, (qv)^2, [s, v], (uv)^2, [p, h], q^h = q^{-1}r^{-1}, r^h = qr, s^h = pq^{-1}r^{-1}s^{-1}t^{-1},t^h = p^{-1}qr^{-1}s^{-1}t, (huv)^2, (hu)^3, p^a=q^{-1}, r^a=s^{-1},t^a=u^{-1}, [v, a], k^a = vk \rangle$ \\ &$L = \langle h, p, q, r, s, t, u, v, k \rangle\cong \mz_3^3.\AGL_2(3)$, $B = \langle p, q, r, s, t, u, v, k \rangle$, $R = \langle p, q, r, s, t, u, v, k, a \rangle$ \\[0.3em]
\hline
\end{tabular}
\end{adjustbox}
\end{table}

For each $(L,B)$ in Table~\ref{tab:amalgams} as a candidate for $(\overline{H}_{\overline{u}}, \overline{H}_{\overline{u}\,\overline{v}})$, we determine the set $\mathcal{X}$ by computation in \textsc{Magma}~\cite{MAGMA}.
Then for each $X\in\mathcal{X}$ as a candidate for $\overline{M}_{\overline{u}}$, we construct $\overline{H}_{\overline{u}}$ as the permutation group induced by the right multiplication action of $L$ on $[L:X]$, where $|L:X|=|\overline{H}_{\overline{u}}|/|\overline{M}_{\overline{u}}|=n$.

Since $\Ga_K$ is $\overline{H}$-arc-transitive, Lemma~\ref{lem:cosetgraph1} implies that $\Ga_K\cong\Cos(\overline{H}, \overline{H}_{\overline{u}}, \overline{H}_{\overline{u}}\overline{h}\overline{H}_{\overline{u}})$, where {$\overline{h}\in\overline{H}$} reverses the arc $(\overline{u},\overline{v})$. We apply \textsc{Magma}~\cite{MAGMA} to find all double cosets $\overline{H}_{\overline{u}}\overline{h}\overline{H}_{\overline{u}}$ with $\overline{h}\in\N_{S_n}(\overline{H}_{\overline{u}\,\overline{v}})$ such that $\overline{h}^2\in \overline{H}_{\overline{u}}$ and $\langle\overline{H}_{\overline{u}},\overline{h}\rangle$ is primitive on $[L:X]$.

Moreover, since $\Ga_K$ is $\overline{M}$-HAT, Lemma~\ref{lem:cosetgraph1} implies that $\Ga_K\cong$$\Cos(\overline{M}, \overline{M}_{\overline{u}}, \overline{M}_{\overline{u}}\{m,m^{-1}\}\overline{M}_{\overline{u}})$, where $m\in\overline{M}$ maps $\overline{u}$ to $\overline{v}$ but does not map $\overline{v}$ back to $\overline{u}$.
Then checking this condition for the candidates, computation in \textsc{Magma}~\cite{MAGMA} shows that, up to isomorphism, $(\overline{H},\overline{M},\overline{H}_{\overline{u}},\overline{M}_{\overline{u}})$ is one of
$(S_5, \AGL_1(5), A_4, \mz_2)$ or $(A_{72}, A_{71}, S_4\times S_3, \mz_2)$, where the former quadruple gives $\Gamma_K\cong \K_{5, 5}-5\K_2$ with $t=2$ while the latter quadruple gives $t=3$.
Note that $H_u\cong\overline{H}_{\overline{u}}$ and $M_u\cong\overline{M}_{\overline{u}}$ as $K$ is semiregular on $V(\Ga)$.
The proof is thus complete.
\end{proof}

We now investigate the {maximal $(\frac{1}{2}, 1)$-pairs of a graph $\Ga$} by the following two lemmas.

\begin{lemma}\label{lem:hat-1t}
Let $\Ga$ be {an} $H$-arc-transitive and $M$-half-arc-transitive graph such that $M<H$ and $M_u^{\Ga(u)}\trianglelefteq H_u^{\Ga(u)}$ for some $u\in V(\Ga)${. T}hen the set of the two $M$-orbits on $A(\Ga)$ is preserved by $H$. Moreover, if $M$ is a maximal subgroup of $H$, then $M$ is normal in $H$ {of index $2$}.
\end{lemma}

\begin{proof}
Let $O_1$ and $O_2$ be the two orbits of $M$ on $A(\Ga)$. For each $v\in V(\Ga)$, define a partition $\{P_1(v), P_2(v)\}$ of $\Ga(v)$ by letting $w\in P_i(v)$ if and only if $(v, w)\in O_i$, where $i=1,2$. Since $M$ is transitive on $V(\Ga)$ and $M_u^{\Ga(u)}\trianglelefteq H_u^{\Ga(u)}$, we have $M_v^{\Ga(v)}\trianglelefteq H_v^{\Ga(v)}$ for every $v\in V(\Ga)$, and so $\{P_1(v),P_2(v)\}$ is a block system of imprimitivity for $H_v^{\Ga(v)}$.

{
Take $h\in H$. To prove that $h$ preserves $\{O_1,O_2\}$, suppose otherwise.
Then there exist arcs $a_1,a_2\in O_1$ such that $a_1^h\in O_1$ and
$a_2^h\in O_2$. Since $O_1$ is an $M$-orbit on $A(\Ga)$, there exists
$m\in M$ such that $a_1^m=a_1^h$. Replacing $h$ by $hm^{-1}$, we may assume
that $a_1^h=a_1$. Note that this replacement does not affect the fact that
some arc in $O_1$ is mapped by $h$ to an arc in $O_2$, since $m^{-1}$ preserves
each $M$-orbit on $A(\Ga)$.}

{
Write $a_1=(u,v)$. Since $h$ fixes the arc $(u,v)$, it fixes both $u$ and $v$,
and hence preserves $P_i(u)$ and $P_i(v)$ for $i=1,2$. Since $\Gamma$ is connected, if an arc in $O_1$ is fixed by $h$, then every arc in $O_1$ must be mapped by $h$ to an arc in the same $M$-orbit. This contradicts the existence of an arc in $O_1$ mapped by $h$ to an arc in $O_2$. Therefore, $h$ preserves $\{O_1,O_2\}$.
}

%Take $g\in H$. For each $v\in V(\Ga)$, pick $m_v\in M$ with $v^{m_v}=v^g$. Then {$g m_v^{-1}\in H_v$.} Since $\{P_1(v),P_2(v)\}$ is a block system for $H_v^{\Ga(v)}$, the element {$g m_v^{-1}$ either preserves each of $P_1(v),P_2(v)$ or interchanges them.} As $m_v\in M$ preserves each $O_i$, {it follows that $g$ either maps $P_i(v)$ to $P_i(v^g)$ for $i=1,2$, or maps $P_i(v)$ to $P_{3-i}(v^g)$ for $i=1,2$. Define $\sigma_g(v)=1$ in the former case and $\sigma_g(v)=-1$ in the latter.}

%Suppose that {$\sigma_g$ is not constant.} Since $\Ga$ is connected, there exists an edge $\{u,w\}$ with $\sigma_g(u)\neq \sigma_g(w)$. Without loss of generality, assume that $(u,w)\in O_1$. Then $(w,u)\in O_2$. Moreover, $(u^g,w^g)\in O_1$ if $\sigma_g(u)=1$, and $(u^g,w^g)\in O_2$ if $\sigma_g(u)=-1$. Similarly, $(w^g,u^g)\in O_2$ if $\sigma_g(w)=1$, and $(w^g,u^g)\in O_1$ if $\sigma_g(w)=-1$. Since $(u^g,w^g)$ and $(w^g,u^g)$ are the two arcs of the same edge and lie in different $M$-orbits, we obtain $\sigma_g(u)=\sigma_g(w)$, a contradiction. Hence $\sigma_g$ is constant on $V(\Ga)$, and so $g$ preserves $\{O_1,O_2\}$.

Moreover, let $H^+$ be the kernel of the action of $H$ on $\{O_1,O_2\}$. Then $M\leq H^+$ and $|H:H^+|\leq 2$. Since $H$ is arc-transitive while $O_1\neq A(\Ga)$, we have $H^+<H$. The maximality of $M$ in $H$ then gives $M=H^+$, so $M\trianglelefteq H$ {of index $2$}.
\end{proof}

\begin{lemma}\label{lem:stab:=2}
Let $\Ga$ be a tetravalent graph, let $u\in V(\Ga)$, let $(M, H)$ be a maximal $(\frac{1}{2},1)$-pair of $\Ga$ such that $M_u^{\Ga(u)}\cong\mz_2$ and $H_u^{\Ga(u)}\cong D_8$, and let $K=\mathrm{Core}_H(M)$. Then $M_u^{\Ga(u)}\ntrianglelefteq H_u^{\Ga(u)}$, $K$ is semiregular on $V(\Ga)$, the order {$|V(\Ga_K)|$ is not a power of $2$, and one of the following holds:
\begin{enumerate}[\rm (a)]
\item\label{lem:stab:=2-a} $K$ is the kernel of $H$ acting on $V(\Ga_K)$, $\Ga$ is a cover of $\Ga_K$, and the pair $(M/K, H/K)$ is a maximal $(\frac{1}{2}, 1)$-pair of $\Ga_K$;
\item\label{lem:stab:=2-b} $K$ is the kernel of $M$ acting on $V(\Ga_K)$, $\Ga_K\cong C_r$, and $M/K\cong D_{2r}$.
\end{enumerate}}
\end{lemma}

\begin{proof}
Since $\Ga$ is $M$-HAT and tetravalent, $M_u$ has two orbits of size $2$ on $\Ga(u)$, and so $M_u^{\Ga(u)}\cong\mz_2$ is semiregular on $\Ga(u)$. Since $M$ is transitive on $V(\Ga)$, $M_v^{\Ga(v)}$ is semiregular on $\Ga(v)$ for every $v\in V(\Ga)$. It follows that each element of $M_v^{[1]}$ fixes all neighbors of $v$, and hence lies in $M_w^{[1]}$ for each $w\in\Ga(v)$. By the connectedness of $\Ga$, $M_u^{[1]}=1$, and consequently $M_u\cong\mz_2$.

Suppose that $M_u^{\Ga(u)}\trianglelefteq H_u^{\Ga(u)}$. By Lemma~\ref{lem:hat-1t}, $M$ is normal in $H$ with $|H:M|=2$. Then $|H_u:M_u|=2$, but it follows from $|H_u|\geq 8$ and $|M_u|=2$ that $|H_u:M_u|\geq 4$, a contradiction.

Thereby we conclude that $M_v^{\Ga(v)}\ntrianglelefteq H_v^{\Ga(v)}$ for each $v\in V(\Ga)$. If $K_v\neq 1$, then it follows from $K_v\leq M_v\cong\mz_2$ that $M_v=K_v$ is normal in $H_v$, a contradiction. Consequently, $K_v=1$ for each $v\in V(\Ga)$, which means that $K$ is semiregular on $V(\Ga)$.

Since $K$ is semiregular on $V(\Ga)$, letting $r$ denote the number of orbits of $K$ on $V(\Ga)$, we have $|V(\Ga)|=r|K|$. The vertex-transitivity of $H$ and $M$ then gives $|H/K|=r|H_u|$ and $|M/K|=2r$.
{Note that the assumption $H_u^{\Ga(u)}\cong D_8$ implies} that $H_u$ is a $2$-group of order at least $8$.
If $r$ is a power of $2$, then $|H/K|=r|H_u|$ is a power of $2$, and $|H/K:M/K|=|H_u|/2\geq4$, whence $M/K$ is not maximal in the $2$-group $H/K$, contradicting the maximality of $M$ in $H$. Therefore, $r$ is not a power of $2$.

Now as $r\geq3$, we conclude from Lemma~\ref{lem:4HAT} that either case~\eqref{lem:stab:=2-a} of Lemma~\ref{lem:stab:=2} holds, or $\Ga_K\cong C_r$. To complete the proof, assume that $\Ga_K\cong C_r$. Let $G$ be the kernel of $H$ acting on $V(\Ga_K)$. Then $H/G\cong D_{2r}$ as $\Ga_K$ is $(H/G)$-arc-transitive. Let $N=M\cap G$ be the kernel of $M$ acting on $V(\Ga_K)$. Then $K\leq N$ and $MG/G\cong M/N$.
If $G\leq M$, then $N=G$, and $M/G$ is a transitive subgroup of $D_{2r}$, whence $|H/G:M/G|\leq 2$. Consequently, $|H:M|=|H/G:M/G|\leq 2$, which gives $|H_u:M_u|\leq 2$, contradicting $|H_u|\geq 8$ and $|M_u|=2$.
Thus, $G\nleq M$ and hence $M<MG$. The maximality of $M$ in $H$ then gives $MG=H$. It follows that $M/N\cong MG/G=H/G\cong D_{2r}$, and so $|M/N|=2r$.
Recalling that $|M|=2r|K|$, we obtain $|N|=|M|/|M/N|=|K|$. As $K\leq N$, this gives $N=K$; that is, $K$ is the kernel of $M$ acting on $V(\Ga_K)$, and $M/K\cong D_{2r}$, leading to case~\eqref{lem:stab:=2-b} of Lemma~\ref{lem:stab:=2}.

\end{proof}

\section{Examples}\label{sec:examples}

We start with an example for case~\eqref{Thmcase:t=1_cover} of Theorem~\ref{Thm} with $K=1$. This example also proves Corollary~\ref{cor-1} and shows that the answer to Question~\ref{que:Sparl} is negative.

\begin{example}\label{exam:hat_in_1-trans_altgra}
Let $L=\langle a,b,c\mid a^2=b^3=c^4=(ab)^8=1,\,c=[a,b]\rangle\cong\mathrm{PGL}_2(7)$, and let
\[
{H}=(L\times L)\rtimes\langle d\rangle\cong\mathrm{PGL}_2(7)\wr S_2,
\]
where $d$ interchanges the coordinates of elements in $L\times L$. Since $c$ is an element of order $4$ in $L'\cong\PSL_2(7)$, we have $\N_{L'}(\langle c\rangle)\cong D_8$. Let
\[
Y=(\N_{L'}(\langle c\rangle)\times\N_{L'}(\langle c\rangle))\rtimes\langle d\rangle\cong D_8\wr S_2,\quad x=(bc^2bc,a^b)
\]
and $\Gamma=\mathrm{Cos}(H,Y,YxY)$.

We first claim that $\Gamma$ is a normal-edge-transitive non-normal Cayley graph of $\mathrm{AGL}_1(7)\times\mathrm{AGL}_1(7)$ of valency $4$. In fact, letting
\[
G=\langle abc,c[b,c]\rangle\times\langle abc,c[b,c]\rangle\cong\mathrm{AGL}_1(7)\times\mathrm{AGL}_1(7),
\]
where $\langle abc,c[b,c]\rangle=\langle c[b,c]\rangle\rtimes\langle abc\rangle\cong\mathrm{AGL}_1(7)$, one has $|G\cap\mathrm{Soc}(H)|=(7\cdot3)^2$. Since $G\cap Y\leq\mathrm{Soc}(H)$ and $|Y|$ is a power of $2$, it follows that $G\cap Y=1$.
Then since $|H|=|G||Y|$, the mapping $\varphi\colon g\mapsto Yg$ is a bijection from $G$ to $[H:Y]$. Moreover, it is straightforward to verify that
\[
S:=\{x,x^{-1},x^d,(x^d)^{-1}\}=G\cap(YxY).
\]
Hence $\varphi$ is a graph isomorphism from $\mathrm{Cay}(G,S)$ to $\Gamma$. Note that the normalizer
\[
M:=\N_H(G)=G\rtimes\langle d\rangle
\]
is edge-transitive and is a maximal subgroup of $H$. Hence $\Gamma\cong\mathrm{Cay}(G,S)$ is a normal-edge-transitive non-normal Cayley graph of valency $4$ on $G$, as claimed.

We now describe $\Alt_M(\Gamma)$. The two $M$-alternating cycles $C_-$ and $C_+$ on $\Gamma$ containing $1\in G$ have vertex sets $\langle x^{-1}x^d\rangle\cup x\langle x^{-1}x^d\rangle$ and $\langle x^dx^{-1}\rangle\cup (x^d)^{-1}\langle x^dx^{-1}\rangle$ respectively. Note that each alternating cycle is determined by its vertex set. Then $V(\Alt_M(\Gamma))=\{C_+^m\mid m\in M\}$.
% and $E(\Alt_M(\Gamma))=\{\{C_1,C_2\}\mid C_1,C_2\in V(\Alt_M(\Gamma)),\, C_1\cap C_2\neq \emptyset\}$.
By computation in~\textsc{Magma}~\cite{MAGMA}, $C_-\cap C_+=\{1_G\}$. Hence $\att_M(\Ga)=1$, and so $\Ga_{\mathcal{B}(M)}=\Ga$ is arc-transitive. In particular,
\[
E(\Alt_M(\Gamma))=\{\{C_1,C_2\}\mid C_1,C_2\in V(\Alt_M(\Gamma)),\, C_1\cap C_2\neq \emptyset\}
\]
can be identified as $G$.

Finally, computation in~\textsc{Magma}~\cite{MAGMA} directly verifies that $\Aut(\Gamma)=H$ and $\Aut(\Alt_M(\Gamma))=M$. Therefore, $\Alt_M(\Ga)$ is half-arc-transitive.
\qed
\end{example}

In the same vein, we give an example for case~\eqref{Thmcase:t=3} of Theorem~\ref{Thm}, which also proves Corollary~\ref{cor-2}.

\begin{example}\label{exam:hat_in_3-trans_non-normal}

Let
\[X=\langle a,b\mid a^2=b^3=(ab)^2=1\rangle\times\langle c,d\mid c^2=d^3=(cd)^4=1\rangle\cong S_3\times S_4,\]
and consider the faithful action $R$ of $X$ on $\Omega := [X : \langle a(cd)^2\rangle]$ by right multiplication. Let $H = \mathrm{Alt}(\Omega) \cong A_{72}$, and let $H_\omega$ be the stabilizer in $H$ of the point $\omega = \langle a(cd)^2\rangle$. Define $Y = R(X' \rtimes \langle ac\rangle) \cong \mz_3 \rtimes S_4$; then $Y$ is a regular subgroup of $H$. Further, let $Z = R((\langle b\rangle \times \langle d\rangle) \rtimes \langle ac^{dc}\rangle) \cong \mz_3^2 \rtimes \mz_2$.

A search using \textsc{Magma}~\cite{MAGMA} reveals an element $x \in N_H(Z) \cap C_H(R(ac^{dc}))$ of order $4$ satisfying the following conditions:
\begin{itemize}
\item $x^2 = R(ac^{dc}) \in Z$;
\item $Y \cap Y^x = Z$;
\item $H = \langle Y, x \rangle$;
\item $Y x Y = Y S$ for some $S = \{g, g^{-1}, g^h, (g^h)^{-1}\}$, where $g \in H_\omega$ and {$h=R(a(cd)^2)$}.
\end{itemize}

{
Therefore, following the same approach as in Example~\ref{exam:hat_in_1-trans_altgra}, the mapping $\varphi:g\mapsto Yg$ from $H_\omega$ to $[H:Y]$ gives a graph isomorphism from $\mathrm{Cay}(H_\omega,S)$ to $\Gamma=\mathrm{Cos}(H, Y, YxY)$. Note that the conjugation by $h$ lies in $\Aut(H_\omega,S)$. This implies that $\Ga$ is a connected tetravalent normal-edge-transitive non-normal Cayley graph on $H_\omega\cong A_{71}$.
Moreover, $\Ga$ is a tetravalent $(H, 3)$-arc-transitive graph with vertex-stabilizer $Y\cong \mz_3\rtimes S_4$.}

Furthermore, let $\Sigma=\Cos(H, R(X), R(X)xR(X))$. Then $\Sigma$ is {an} $(H, 3)$-arc-transitive graph with vertex-stabilizer $R(X)\cong S_3\times S_4$ and $H_\omega$-HAT with vertex-stabilizer $R(\langle a(cd)^2\rangle)\cong \mz_2$.
\qed
\end{example}

Below we give {an example for} part (\ref{Thmcase:t=2}) of Theorem~\ref{Thm}.

\begin{example}\label{exam:hat_in_2-trans}
Let $\Ga=\K_{5, 5}-5\K_2$. Then $\Aut(\Ga)\cong S_5\times\mz_2$ has a subgroup $H\cong S_5$ such that $\Ga$ is $(H, 2)$-arc-transitive with {vertex-stabilizers} isomorphic to $A_4$, and $H$ has a {core-free maximal subgroup $M\cong\AGL(1,5)$} such that $\Ga$ is $M$-HAT with {vertex-stabilizers} isomorphic to $\mz_2$.
\qed
\end{example}

Finally, the following is an example for case~\eqref{Thmcase:t=1_cycle} of Theorem~\ref{Thm} {by taking $(M,H)=(\N_{\Aut(\Ga)}(R_G(G)),\Aut(\Ga))$.}

\begin{example}\label{exam:hat_in_1-trans_cycle}
Let $\Ga=\Cay(G,S)$ with
\begin{align*}
G=\langle a,b,c,d\mid\,&a^9=b^3=c^3=d^3=[[b,c],b]=[[b,c],c]=[[b,c],d]=1,\\
&a^{-1}ba=c,\; a^{-1}ca=d,\; a^{-1}da=b[c,d]\rangle
\end{align*}
and $S=\{ab, ab^{-1}, (ab)^{-1}, (ab^{-1})^{-1}\}$. According to~\textsc{Magma}~\cite{MAGMA} computation, $|G|=3^8$, $\Aut(\Ga)_1\cong D_8$, and $\N_{\Aut(\Ga)}(R_G(G))\cong R_G(G)\rtimes\mz_2$ is maximal in $\Aut(\Ga)$. Moreover, $\Ga$ is $(\Aut(\Ga), 1)$-transitive and $\N_{\Aut(\Ga)}(R_G(G))$-HAT. Let $K$ be the core of $\N_{\Aut(\Ga)}(R_G(G))$ in $\Aut(\Ga)$. Then $|R_G(G): K|=3$ and $\Ga_K\cong C_3$.
\qed
\end{example}

\section{Proof of main results}\label{sec:sum}

Corollaries~\ref{cor-1} and~\ref{cor-2} follow from Examples~\ref{exam:hat_in_1-trans_altgra} and~\ref{exam:hat_in_3-trans_non-normal}, respectively. Now we prove Theorem~\ref{Thm}.
Let $\Ga$ be a tetravalent graph, let $(M, H)$ be a maximal $(\frac{1}{2}, t)$-pair of $\Ga$ with $t\geq1$, let $K=\mathrm{Core}_H(M)$, and let $u\in V(\Ga)$.
If $t\geq2$, then Lemma~\ref{lem:hat-2at} asserts that~\eqref{Thmcase:t=3} or~\eqref{Thmcase:t=2} holds.
Assume that $t=1$. If $M_u^{\Ga(u)}\cong\mz_2$ and $H_u^{\Ga(u)}\cong D_8$, then by Lemma~\ref{lem:stab:=2}, {$M_u^{\Ga(u)}\ntrianglelefteq H_u^{\Ga(u)}$, $K$ is semiregular on $V(\Ga)$, $|V(\Ga_K)|$ is not a power of 2,} and we have~\eqref{Thmcase:t=1_cover} or~\eqref{Thmcase:t=1_cycle}. If either $M_u^{\Ga(u)}\ncong\mz_2$ or $H_u^{\Ga(u)}\ncong D_8$, then since $(M, H)$ is a $(\frac{1}{2}, 1)$-pair of $\Ga$, it follows that $M_u^{\Ga(u)}$ is normal in $H_u^{\Ga(u)}$ with $|H_u^{\Ga(u)}:M_u^{\Ga(u)}|=2$. In this case, Lemma~\ref{lem:hat-1t} implies that $M$ is normal in $H$, as in case~\eqref{Thmcase:t=1} of Theorem~\ref{Thm}. Moreover, according to Examples~\ref{exam:hat_in_1-trans_altgra},~\ref{exam:hat_in_3-trans_non-normal},~\ref{exam:hat_in_2-trans} and~\ref{exam:hat_in_1-trans_cycle}, examples for each of the cases~\eqref{Thmcase:t=3},~\eqref{Thmcase:t=2},~\eqref{Thmcase:t=1_cover} and~\eqref{Thmcase:t=1_cycle} of Theorem~\ref{Thm} exist. This completes the proof.

\section*{Acknowledgments}

This work was supported by the National Natural Science Foundation of China (12425111, 12331013, 12161141005).
The authors would like to thank Primo\v z \v Sparl for a useful discussion on the topic. The authors also thank the anonymous referees for their helpful comments and suggestions.

\section*{Appendix}

\tiny
\section*{Code for Lemma~3.1}

\begin{verbatim}
// For each 2-arc-transitive amalgam (Hu, Huv) from Table 1, we search for
// tetravalent graphs Gamma with a maximal (1/2, t)-pair (M, H), t >= 2.
// Step 1: Find candidate HAT stabilizers Mu < Hu.
// Step 2: Find arc-reversing elements h in N_{Sym}(Huv).
// Step 3: Build H = <Hu, h> and check primitivity.
// Step 4: Verify M = Stab_H(1) admits a forward element.

// ---- Amalgam data (Table 1) ----
Amalgams := [];
L<x,y,s> := Group<x,y,s | x^2, y^2, s^3, (x,y), x^s*y^-1, y^s*(x*y)^-1>;
Append(~Amalgams, <L, sub<L|s>>);                // [1] A4
L<x,y,s,t> := Group<x,y,s,t | x^2, y^2, s^3, t^2, (x,y), s^t*s, x^s*y^-1, y^s*(x*y)^-1, x^t*y^-1>;
Append(~Amalgams, <L, sub<L|s,t>>);              // [2] S4
L<x,y,c,d> := Group<x,y,c,d | x^2, y^2, c^3, d^3, (x,y), (c,d), (c,x), (c,y), x^d*y^-1, y^d*(x*y)^-1>;
Append(~Amalgams, <L, sub<L|c,d>>);              // [3] C3 x A4
L<x,y,c,d,t> := Group<x,y,c,d,t | x^2, y^2, c^3, d^3, t^2, (x,y), (c,d), (c,x), (c,y), c^t*c, d^t*d,
    x^d*y^-1, y^d*(x*y)^-1, x^t*y^-1>;
Append(~Amalgams, <L, sub<L|c,d,t>>);            // [4] C3 : S4
L<x,y,c,d,r,s> := Group<x,y,c,d,r,s | x^2, y^2, c^3, d^3, r^2, s^2, (x,y), (c,d), (r,s), (c,x), (c,y), c^r*c,
    (d,r), (c,s), d^s*d, x^d*y^-1, y^d*(x*y)^-1, x^s*y^-1, (r,x), (r,y)>;
Append(~Amalgams, <L, sub<L|c,d,r,s>>);          // [5] S3 x S4
L<t,x,y,c,d,e> := Group<t,x,y,c,d,e | t^2, c^3, d^3, e^3, x^2, y^2, (c,d), (c,e), (d,e)*c^-1, (x,y), (c*x)^2,
    (d*x)^2, (e,x), (c*y)^2, (d,y), (e*y)^2, c^t*d, y*(e*t)^2*e^-1*t*e^-1, (e*t)^4*x>;
Append(~Amalgams, <L, sub<L|x,y,c,d,e>>);        // [6] 4-AT
L<h,p,q,r,s,t,u,v,k> := Group<h,p,q,r,s,t,u,v,k |
    h^4, p^3, q^3, r^3, s^3, t^3, u^3, v^2, k^2, k*h^2, (p,q),(p,r),(p,s),(p,t),(p,u),(q,r),(q,s),(q,t),(q,u),
    (r,s),(r,t),(u,s), (s,t)*p^-1, (u,r)*q^-1, (t,u)*(q*r*s*p^-1)^-1, (k,v),(t*k)^2,(r*k)^2,(p,k),(q*k)^2,(s*k)^2,(u,k),
    (t*v)^2,(r,v),(p*v)^2,(q*v)^2,(s,v),(u*v)^2, (p,h), q^h*(q^-1*r)^-1, r^h*(q*r)^-1, s^h*(p*q^-1*r^-1*s^-1*t^-1)^-1,
    t^h*(p^-1*q*r^-1*s^-1*t)^-1, (h*u*v)^2, (h*u)^3>;
Append(~Amalgams, <L, sub<L|p,q,r,s,t,u,v,k>>);  // [7] 7-AT

// ---- Step 1: Find HAT stabilizer candidates ----
// Returns subgroups Mu of Hu with two orbits of size 2 on
// Gamma(u) = [Hu : Huv] and core-free in Hu.
function FindHATStabilizers(Hu, Huv)
    f := CosetAction(Hu, Huv);
    cands := [];
    for s in Subgroups(Hu : OrderDividing := #Sylow(Hu,2) div 2) do
        Mu := s`subgroup;
        O := Orbits(f(Mu));
        if #O ne 2 or #O[1] ne 2 then continue; end if;
        _, _, K := CosetAction(Hu, Mu);
        if #K ne 1 then continue; end if;
        Append(~cands, Mu);
    end for;
    return cands;
end function;

// ---- Step 2: Find arc-reversing element candidates ----
// Returns elements h in N_{Sym}(Huv) with 2-power order,
// h not in Hu, h^2 in Hu.
function FindArcReversors(Hu_perm, Huv_perm)
    Nuv := SymmetricNormalizer(Huv_perm);
    return [n : n in Nuv | IsPowerOf(Order(n), 2) and n notin Hu_perm and n^2 in Hu_perm];
end function;

// ---- Step 3: Check forward element condition ----
// m maps u to v but does NOT map v back to u.
function IsForwardElement(m, Mu_perm)
    T := Transversal(Mu_perm, Mu_perm meet Mu_perm^m);
    return not exists{t : t in T | m*t*m in Mu_perm};
end function;

// ---- Main search ----
function SearchPairs(Hu, Huv)
    results := [];
    for Mu in FindHATStabilizers(Hu, Huv) do
        phi := CosetAction(Hu, Mu);
        Hu_p := phi(Hu); Huv_p := phi(Huv); Mu_p := phi(Mu);
        Omega := Sym(GSet(Hu_p));
        for h in FindArcReversors(Hu_p, Huv_p) do
            H := sub<Omega | Hu_p, h>;
            if not IsPrimitive(H) then continue; end if;
            if #Core(H, Hu_p) ne 1 then continue; end if;
            M := Stabilizer(H, 1);
            if #Core(M, Mu_p) ne 1 then continue; end if;
            // Step 4: search for a forward element m in M
            for i in Hu_p do
                m := i * h;
                if m notin M then continue; end if;
                if not IsForwardElement(m, Mu_p) then continue; end if;
                if #sub<M | Mu_p, m> eq #M then
                    Append(~results, <H, M, Hu_p, Mu_p>);
                    break; end if;
            end for;
        end for;
    end for;
    return results;
end function;

// ---- Run for each amalgam ----
for i in [1..#Amalgams] do
    L := Amalgams[i][1]; B := Amalgams[i][2];
    Hu, psi := PermutationGroup(L);
    printf "Amalgam [%o] %o: ", i, GroupName(Hu);
    RL := SearchPairs(Hu, psi(B));
    if #RL eq 0 then printf "no valid pair\n";
    else quads := {@ <r[1],r[2],r[3],r[4]> : r in RL @};
        for q in quads do
            printf "(H,M,Hu,Mu) = (%o,%o,%o,%o)\n",
                GroupName(q[1]),GroupName(q[2]),
                GroupName(q[3]),GroupName(q[4]);
        end for;
    end if;
end for;
/* OUTPUT:
Amalgam [1] A4:     (S5, F5, A4, C2)
Amalgam [2..4,6,7]: no valid pair
Amalgam [5] S3*S4:  (A72, A71, S3*S4, C2)
*/
\end{verbatim}

\medskip

\section*{Code for Example~4.1}

\begin{verbatim}
// Construct Gamma = Cos(H, Y, YxY), H = PGL(2,7) wr S_2, Y = D_8 wr S_2.
// Verify Gamma is a normal-edge-transitive non-normal Cayley graph on
// G = AGL(1,7) x AGL(1,7), and Alt_M(Gamma) is half-arc-transitive.

// ---- Part 1: Group construction ----
L<a,b,c> := Group<a,b,c | a^2, b^3, c^4, (a*b)^8, c^-1*(a,b)>;
K := DerivedGroup(L);   // L' = PSL(2,7)
PL, phi := PermutationGroup(L);
X, SeqDeg, inc, proj := WreathProduct(PL, Sym(2));
a1 := SeqDeg[1](phi(a));  b1 := SeqDeg[1](phi(b));  c1 := SeqDeg[1](phi(c));
d := inc(Sym(2)!(1,2));               // coordinate swap

// ---- Part 2: Subgroups Y and G ----
// Y = N_{L'}(<c>) wr S_2 = D_8 wr S_2  (vertex stabilizer)
NKc := Normalizer(K, sub<L|c>);
Y := sub<X | SeqDeg[1](phi(NKc)), SeqDeg[2](phi(NKc)), d>;
assert #Y eq 128;

// G = AGL(1,7) x AGL(1,7)  (regular subgroup)
AGL17 := sub<L | a*b*c, c*(b,c)>;
G := sub<X | SeqDeg[1](phi(AGL17)), SeqDeg[2](phi(AGL17))>;
assert #G eq 42^2 and #(G meet Y) eq 1;  // G regular on [H:Y]

// ---- Part 3: Cayley graph Gamma ----
x := b1*c1^2*b1*c1*(a1^b1)^d;
S := {x, x^-1, x^d, (x^d)^-1};
assert #S eq 4;
YxY := {yl*x*yr : yl, yr in Y};
assert (YxY meet Set(G)) eq S;  // S = G cap YxY
GG := sub<G | SetToSequence(S)>;
assert #GG eq #G;
gra := CayleyGraph(GG : Labelled := false, Directed := false);
assert IsConnected(gra);

// ---- Part 4: Verify Aut(Gamma) = H ----
A := AutomorphismGroup(gra);
flag, iso := IsIsomorphic(X, A);
assert flag;   // Aut(Gamma) ~ PGL(2,7) wr S_2

// M = N_{Aut}(G) = G : <d>, maximal HAT group
N := Normalizer(A, iso(G));
assert N eq iso(sub<X | G, d>);
assert IsMaximal(A, N);
"Gamma: |V|=", Order(gra), "|M|=", #N, "|Aut|=", #A;

// ---- Part 5: Graph of M-alternating cycles ----
HATGrp := N;
GV := GSet(HATGrp);
VV := Vertices(gra);
u := Rep({Index(VV, w) : w in Neighbors(VV[1])});
Arcs := GSet(HATGrp, Orbit(HATGrp, <GV[1], {GV[u]}>));
D := Digraph<GV | Arcs>;
VD := Vertices(D);

// Trace one alternating cycle
v := VD[1]; AltCyc := [v]; flag := "tail";
while true do
    if flag eq "tail" then
        nbr := OutNeighbors; flag2 := "head";
    else nbr := InNeighbors; flag2 := "tail"; end if;
    N1 := nbr(AltCyc[1]) diff Seqset(AltCyc);
    if IsEmpty(N1) then break; end if;
    N2 := nbr(AltCyc[#AltCyc]) diff Seqset(AltCyc) diff {Rep(N1)};
    if IsEmpty(N2) then AltCyc := [Rep(N1)] cat AltCyc; break; end if;
    AltCyc := [Rep(N1)] cat AltCyc cat [Rep(N2)];
    flag := flag2;
end while;
altcyc := {GV[Index(VD, w)] : w in AltCyc};
AC := Orbit(HATGrp, altcyc);

// Attachment number
att := 0;
for C1, C2 in AC do
    s := #(C1 meet C2);
    if s ne 0 and s ne #Rep(AC) then att := s; break C1; end if;
end for;
assert att eq 1;
"Alt cycles:", #AC, "length:", #Rep(AC), "att:", att;

// Build Alt_M(Gamma) and verify it is HAT
EdgeAlt := GSet(HATGrp,
    {{C1, C2} : C1, C2 in AC | #(C1 meet C2) eq att});
AltGraph := Graph<AC | EdgeAlt>;
AutAlt := AutomorphismGroup(AltGraph);
"Alt_M: edge-trans=", IsEdgeTransitive(AltGraph), "arc-trans=", IsSymmetric(AltGraph);
assert IsEdgeTransitive(AltGraph) and not IsSymmetric(AltGraph);
// Conclusion: Alt_M(Gamma) is HAT (negative answer to Question 5.8).
\end{verbatim}

\medskip
\section*{Code for Example~4.2}

\begin{verbatim}
// Construct X = S3 x S4, act on Omega = [X : <a(cd)^2>], |Omega| = 72.
// Build H = Alt(Omega) ~= A_72. Find element x satisfying four conditions from Example 4.2.
// Verify Gamma = Cos(H, Y, YxY) ~= Cay(H_w, S) is a connected tetravalent (H,3)-arc-transitive non-normal Cayley graph,
// and Sigma = Cos(H, R(X), R(X)xR(X)) is H_w-HAT with M_u ~= Z2.

// ---- Part 1: Group setup ----
X1<a,b> := Group<a,b | a^2, b^3, (a*b)^2>;
X2<c,d> := Group<c,d | c^2, d^3, (c*d)^4>;
X<a,b,c,d> := DirectProduct(X1, X2);
assert #X eq 6*24;   // |S3 x S4| = 144

R := CosetAction(X, sub<X | a*(c*d)^2>);
assert Degree(R(X)) eq 72;       // |Omega| = 72

// ---- Part 2: H = Alt(Omega), subgroups Y and Z ----
Hs := Sym(GSet(R(X)));
H := NormalSubgroups(Hs)[2]`subgroup;  // Alt(Omega) ~ A_72
Hw := Stabilizer(H, 1);          // H_omega
h := R(a*(c*d)^2);

Y := sub<R(X) | R(DerivedSubgroup(X)), R(a*c)>;
assert #Y eq 72 and IsTransitive(Y) and #Stabilizer(Y, 1) eq 1;
Z := R(sub<X | b, d, a*c^(d*c)>);
assert #Z eq 18;

// ---- Part 3: Search for x and Cayley set S ----
ac_dc := R(a*c^(d*c));
assert Order(ac_dc) eq 2;
NC := Centralizer(Normalizer(H, Z), ac_dc);

RL := [];
for xx in NC do
    if Order(xx) ne 4 or xx^2 ne ac_dc then continue; end if;
    if (Y meet Y^xx) ne Z then continue; end if;
    if #sub<H | Y, xx> ne #H then continue; end if;
    Yx := {@ i*xx : i in Y @};
    YxY := Set({@ i*xx*j : i, j in Y @});
    for g in Yx do
        if g notin Hw or g^-1 notin YxY then continue; end if;
        S := {g, g^-1} join {g^h, (g^h)^-1};
        if #S ne 4 then continue; end if;
        if {i*s : i in Y, s in S} ne YxY then continue; end if;
        Append(~RL, <xx, g>); break xx;
    end for;
end for;
assert #RL ge 1;
x := RL[1][1]; g := RL[1][2];
S := {g, g^-1, g^h, (g^h)^-1};
assert #S eq 4 and g in Hw;

// ---- Part 4: Verify Gamma = Cay(H_w, S) ----
assert forall{s : s in S | s in Hw};
assert #sub<Hw | SetToSequence(S)> eq #Hw;   // connected
assert IsSimple(Hw) and not IsNormal(H, Hw); // H_w ~= A_71
assert h in Hw and {s^h : s in S} eq S;      // normal-edge-transitive

// ---- Part 5: Sigma = Cos(H, R(X), R(X)xR(X)) ----
RX := R(X);
assert #RX eq 144;
D := Set({@ i*x*j : i, j in RX @});
assert D eq {d^-1 : d in D} and #D div #RX eq 4;  // val 4
Mu := Hw meet RX;
assert #Mu eq 2 and Mu eq R(sub<X | a*(c*d)^2>);   // Z2
// HAT: Mu has 2 orbits of size 2 on neighbor cosets
reps := [];
for dd in D do
    if forall{r : r in reps | dd * r^-1 notin RX} then
        Append(~reps, dd);
    end if;
    if #reps eq 4 then break; end if;
end for;
mu := Rep({m : m in Mu | m ne Id(Mu)});
perm := [];
for i in [1..4] do
    for j in [1..4] do
        if reps[i]*mu*reps[j]^-1 in RX then
            Append(~perm, j); break;
        end if;
    end for;
end for;
assert #perm eq 4;
assert forall{i : i in [1..4] | perm[perm[i]] eq i and perm[i] ne i};

// ---- Part 6: H_w = <h, xw> ----
xw := Random({@ r*x : r in RX | r*x in Hw @});
assert #sub<Hw | h, xw> eq #Hw;
\end{verbatim}

\medskip
\section*{Code for Example~4.3}

\begin{verbatim}
// Verify: Gamma = K_{5,5} - 5K_2 has H ~= S5 with Gamma being
// (H,2)-arc-transitive (H_u ~= A4), and a core-free maximal
// subgroup M ~= AGL(1,5) with Gamma being M-HAT (M_u ~= Z_2).

// ---- Construct Gamma = K_{5,5} - 5K_2 ----
Ga := Graph<10 | {{i, j+5} : i in [1..5], j in [1..5] | i ne j}>;
V := Vertices(Ga);
assert #V eq 10 and Degree(V[1]) eq 4 and IsConnected(Ga);
A := AutomorphismGroup(Ga);   // Aut ~= S5 x C2

// ---- Find H ~= S5: vertex-transitive, (H,2)-arc-transitive ----
for s in Subgroups(A : OrderEqual := 120) do
    if IsTransitive(s`subgroup) then H := s`subgroup; break; end if;
end for;
Hu := Stabilizer(H, 1);
assert #H eq 120 and #Hu eq 12;  // |S5|=120, |A4|=12
nbrs := {Index(w) : w in Neighbors(V[1])};
j := Rep(nbrs);
assert #Orbit(Hu, j) eq 4;             // arc-transitive
Huv := Stabilizer(Hu, j);
nbrs_v := {Index(w) : w in Neighbors(V[j])} diff {1};
assert #Orbit(Huv, Rep(nbrs_v)) eq #nbrs_v;  // 2-arc-transitive

// ---- Find M ~= AGL(1,5): core-free maximal in H, M-HAT ----
for s in MaximalSubgroups(H) do
    if IsTransitive(s`subgroup) and #Core(H, s`subgroup) eq 1 then
        M := s`subgroup; break;
    end if;
end for;
Mu := Stabilizer(M, 1);
assert #M eq 20 and #Mu eq 2;  // |AGL(1,5)|=20, |Z_2|=2
assert #Orbit(Mu, j) eq 2;     // M is HAT
assert IsMaximal(H, M);
\end{verbatim}

\medskip
\section*{Code for Example~4.4}

\begin{verbatim}
// Cayley graph Cay(G, S) with |G| = 3^8
G<a,b,c,d> := Group<a,b,c,d | a^9, b^3, c^3, d^3, (b,c,b), (b,c,c), (b,c,d),
    a^-1*b*a*c^-1, a^-1*c*a*d^-1, a^-1*d*a*(b*(c,d))^-1>;
"Order of G:", #G;   // 6561

// Cayley graph via regular representation
phi, Q := CosetAction(G, sub<G|Id(G)>);
n := Degree(Q);
s1 := phi(a*b); s2 := phi(a*b^-1);
SS := {s1, s2, s1^-1, s2^-1};
Ga := Graph<n | [{i^s : s in SS} : i in [1..n]]>;
V := Vertices(Ga);
"Vertices:", #V, " Valency:", Degree(V[1]), " Connected:", IsConnected(Ga);

A := AutomorphismGroup(Ga);
A1 := Stabilizer(A, 1);
"Aut order:", #A, " Aut_1:", GroupName(A1);   // 52488, D4

// R_G(G) as Sylow 3-subgroup of Aut(Gamma)
R := SylowSubgroup(A, 3);
"R regular:", IsTransitive(R) and #Stabilizer(R,1) eq 1;
"R ~= G:", IsIsomorphic(R, Q);

// M = N_{Aut}(R_G(G))
M := Normalizer(A, R);
M1 := Stabilizer(M, 1);
"M order:", #M, " M/R:", GroupName(quo<M|R>);   // 13122, C2
"M maximal:", IsMaximal(A, M);

// Verify: A is arc-transitive, M is HAT
nbrs := {Index(w) : w in Neighbors(V[1])};
j := Rep(nbrs);
"A arc-trans:", #Orbit(A1, j) eq 4;
"M HAT:", #Orbit(M1, j) eq 2;

// Core K and quotient graph
K := Core(A, M);
"|R:R cap K|:", Index(R, R meet K);   // 3
"#K-orbits:", #Orbits(K);            // 3 (Gamma_K ~= C_3)
\end{verbatim}

\end{document}